\documentclass[12pt,one,english]{amsart}
\usepackage{mathabx}
\usepackage{braket}
\usepackage{bbm}
\usepackage[T1]{fontenc}
\usepackage{amsfonts}
\usepackage{mathrsfs}
\usepackage{stmaryrd}
\usepackage{upgreek}
\usepackage{textcomp}
\usepackage[utf8]{inputenc}

\usepackage{extarrow}
\usepackage{enumerate}
\usepackage{enumitem}
\usepackage{braket}
\usepackage{graphicx}
\usepackage[colorlinks=true, linkcolor=blue]{hyperref}
\usepackage{latexsym}
\usepackage{mathtools}
\usepackage{mathtext}
\usepackage{amsmath}
\usepackage{amssymb}
\usepackage{amsthm}
\usepackage{tikz}
\usepackage{tikz-cd}
\usepackage[all,cmtip]{xy}
\usepackage[mathscr]{eucal}
\usepackage[toc,page]{appendix}
\usepackage{a4wide}
\usepackage[
backend=biber,
style=alphabetic-verb,
isbn=false, doi=false, url=false
]{biblatex}
\AtEveryBibitem{%
  \clearfield{addendum}%
  \clearfield{howpublished}%
}

\usepackage{csquotes}
\MakeOuterQuote{"}

\usepackage[colorinlistoftodos]{todonotes}
\setuptodonotes{fancyline}
\usepackage{setspace} 

\usetikzlibrary{arrows}
\usetikzlibrary{matrix}
\usetikzlibrary{shapes}
\usetikzlibrary{snakes}
\usetikzlibrary{matrix}

\DeclareMathOperator{\Spec}{\textup{Spec}}

\usepackage{xcolor}
\theoremstyle{plain}
\newtheorem{thm}{Theorem}[section]
\newtheorem*{thm*}{Theorem}
\newtheorem{lem}[thm]{Lemma}

\newtheorem{prop}[thm]{Proposition}

\theoremstyle{remark}
\newtheorem{rmk}[thm]{Remark}

\newtheorem*{rmk*}{Remark}
\newtheorem{ex}[thm]{Example}

\theoremstyle{definition}
\newtheorem{defn}{Definition}[section] 

\newtheorem*{const*}{Construction}
\newtheorem{conv}[defn]{Conventions}

\theoremstyle{plain}
\newtheorem{thmI}{Theorem}

\newcommand{\sE}{{\mathcal E}}
\newcommand{\sF}{{\mathcal F}}
\newcommand{\sG}{{\mathcal G}}

\newcommand{\sI}{{\mathcal I}}

\newcommand{\Isoc}{{\rm Isoc}}

\AtEveryBibitem{\clearlist{language}}

\begin{document}
\title{Isocrystals on unibranch varieties and open immersions}

\author{Adrian Langer}
\email{alan@mimuw.edu.pl}
\address{University of Warsaw, Institute of Mathematics,
	ul.\ Banacha 2, 02-097 Warszawa, Poland}
	
\author{Lei Zhang}
\email{cumt559@gmail.com}
\address{Sun Yat-Sen University\\
    School of Mathematics (Zhuhai)\\    
    Zhuhai, 
    Guangdong Province\\ China}

\date{\today}

\begin{abstract}
We prove that for a connected, geometrically unibranch variety \(X\) over a
perfect field of characteristic \(p>0\) and a dense open subset \(U\subset X\),
the canonical homomorphism between the Tannaka duals of the categories of
overconvergent isocrystals is a quotient map. 
We also show that the analogous statement for convergent isocrystals, and for
all isocrystals, fails; the counterexample is based on Crew's construction of a unit‑root \(F\)-isocrystal that is convergent but not overconvergent.
\end{abstract}

\maketitle


\section*{Introduction}

Let $X$ be a unibranch complex analytic variety and let $U\subset X$ be the complement of a proper closed analytic subset. Then we have a surjective map $\pi_1^{\sf top}(U) \to \pi_1^{\sf top}(X)$
of topological fundamental groups of $U $ and $X$ (see, e.g., \cite[(0.7) (B)]{Fulton-Lazarsfeld1981}). 
An analogue of this theorem  in positive characteristic for the fundamental group scheme coming from $F$-divided bundles was proven in \cite{Langer-Zhang2025}. The main aim of this note is to check if an analogous result holds in positive characteristic in case of fundamental group schemes coming from isocrystals, convergent isocrystals and overconvergent isocrystals.

More precisely, let \(k\) be a perfect field of characteristic \(p>0\), and let \(K\) be a
complete discretely valued field of mixed characteristic \((0,p)\) with the valuation ring  \(V\)
and \(k\) as the residue field of  \(V\). For a separated \(k\)-scheme of finite type \(X\), we denote by
\(\Isoc(X/K)\) (resp. \(\Isoc^{\dagger}(X/K)\),
\(\Isoc^{\dagger\dagger}(X/K)\)) the category of isocrystals (resp. convergent,
overconvergent isocrystals) on \(X/K\). If \(X\) is connected and admits a \(k\)-rational point, the convergent and overconvergent categories are neutral Tannakian. If \(X\) is moreover regular, the same holds for \(\Isoc(X/K)\).
We write $\pi_1(X/K)$, $\pi_1^{\dagger}(X/K)$ and $\pi_1^{\dagger\dagger}(X/K)$
for the corresponding fundamental group schemes, suppressing the choice of a base point from the notation.
Then we prove the following result { (cf.~ Theorem \ref{relative gerbe} for a more general
statement):}

\begin{thmI}\label{main1} Assume that $X$ is a connected, geometrically
    unibranch, separated $k$-scheme of finite type and let	$U \hookrightarrow
    X$ be a dense open subscheme. {If $U(k)\neq\emptyset$,} then:
	\begin{enumerate}
		\item the canonical homomorphism  $\pi_1^{\dagger\dagger}(U/K)\to \pi_1^{\dagger\dagger}(X/K)$ is  faithfully flat,
		\item the canonical homomorphisms 
		$\pi_1(U/K)\to \pi_1(X/K)$ and $\pi_1^{\dagger}(U/K)\to
        \pi_1^{\dagger}(X/K)$ need not be faithfully flat even when $X$ is
        smooth.
	\end{enumerate}
\end{thmI}

In case $X$ is
smooth (1) follows from \cite[Thm.~5.2.1 and Prop.~5.3.1]{Ked07}. In general, (1) follows from this by using alterations, some descent and the results of N. Tsuzuki (see \cite{Tsu12}). The counterexample in (2) is based on a
classical example due to R. Crew (see \cite[Rmk.~4.15]{Crew87}; more
straightforward arguments appear in \cite[Rmk.~5.3.2]{Ked07} and \cite[Example~4.6
and Rmk.~5.12]{Kedlaya2022}).

\section{Preliminaries}
Throughout, \(k\) denotes a perfect field of characteristic \(p>0\), and \(K\) a complete
discretely valued field of mixed characteristic \((0,p)\) with valuation ring \(V\) and uniformizer \(\pi\).
A \emph{variety} is a separated \(k\)-scheme of finite type.
For isocrystals we always assume that \(v(p)\le p-1\);
this condition guarantees that the ideal \((\pi)\subset V\) possesses a unique PD‑structure
(see~\cite[Example~3.2]{BO78}).
When discussing \(F\)-isocrystals we moreover fix a lift \(\sigma\colon V\cong
V\) of the absolute Frobenius of \(k\). {Set $K_0\coloneqq \set{x\in
K|\sigma(x)=x}$.}

For a variety \(X\), we denote by
\begin{itemize}
\item \(\Isoc(X/K)\) (resp. \(\Isoc^{\dagger}(X/K)\),
    \(\Isoc^{\dagger\dagger}(X/K)\)) the category of isocrystals (resp. convergent,
    overconvergent isocrystals) on \(X/K\),
\item $F$-\(\Isoc^{}(X/K)\) (resp. $F$-\(\Isoc^{\dagger}(X/K)\),
    $F$-\(\Isoc^{\dagger\dagger}(X/K)\)) the category of $F$-isocrystals (resp. convergent,
    overconvergent $F$-isocrystals) on \(X/K\).
\end{itemize}
We will write ($F$)-isocrystals or ($F$)-$\Isoc(X/K)$ to mean isocrystals
\emph{and} $F$-isocrystals. Similarly, we will write (convergent) isocrystals,
meaning convergent isocrystals \emph{and} isocrystals.

All convergent/overconvergent isocrystal categories are \(K\)-linear abelian rigid tensor categories
(cf.~\cite[2.9]{Ogus84}, \cite[2.1.9, 2.2.3]{Berthelot96}). If $X$ is regular, then so
is the category of isocrystals (cf.~\cite[{Corollary
2.4}]{drinfeld24}). {If $X$ is connected, then} these are Tannakian categories over a finite field
extension $L/K$. If $X$ is geometrically connected, then these are
$K$-Tannakian categories. Since every $L$-Tannakian category corresponds
uniquely to an affine gerbe over $L$, we will denote $\Pi_{X}^{\Isoc^\dagger}$
the affine $L$-gerbe corresponding to $\Isoc^{\dagger}(X/K)$, and similarly for
other types of isocrystals.

{An \(F\)-isocrystal is a pair \((E,\Phi_E)\), where \(E\) is an isocrystal
and
 \(
 \Phi_E:F_X^*E\xrightarrow{\sim}E.
 \)
 A morphism \(u:(E,\Phi_E)\to(E',\Phi_{E'})\) is a morphism \(u:E\to E'\)
 satisfying
 \(
 u\circ\Phi_E=\Phi_{E'}\circ F_X^*u.
 \)
 Consequently, the categories of \(F\)-isocrystals are naturally
\(K_0\)-linear, but need not be \(K\)-linear. If $\Isoc^{}(X/K)$ is an
$L$-Tannakian category, the Frobenius map induces a $\sigma$-compatible automorphism
\(\sigma_L\) of \(L\), 
so the
corresponding $F$-isocrystal category $F$-$\Isoc(X/K)$ is a Tannakian category
over  $L_0\coloneqq\set{x\in L|\sigma_L(x)=x}$. Similar conclusions hold for convergent/overconvergent
$F$-isocrystals. }

For any map of $k$-varieties
\(f\colon Y\rightarrow X\), we have restriction functors
\begin{align*}
& f^*\colon (F)\text{-}\Isoc(X/K)\longrightarrow (F)\text{-}\Isoc(Y/K)\\
& f^{\dagger}\colon(F)\text{-}\Isoc^{\dagger}(X/K)\longrightarrow (F)\text{-}\Isoc^{\dagger}(Y/K)\\
& f^{\dagger\dagger}\colon(F)\text{-}\Isoc^{\dagger\dagger}(X/K)\longrightarrow(F)\text{-}\Isoc^{\dagger\dagger}(Y/K)
\end{align*}

\begin{lem}[Faithfulness of pullback]
\label{lem:pullback-faithful}
Suppose that \(X\) is connected and $Y$ is nonempty. Then the above functors
\(
f^{\dagger}
\) and $f^{\dagger\dagger}$
are faithful. If $X$ is regular, then $f^*$ is faithful as well.
\end{lem}
\begin{proof} Let's prove the claim for convergent isocrystals, the others are
    completely analogous. Replacing $Y$ by a closed point, we may and do assume
    that $Y=\Spec(l)$, where $l$ is a finite field extension of $k$. In this
    case, both $\Isoc^\dagger(X/K)$ and $\Isoc^\dagger(Y/K)$ are Tannakian
    categories, but perhaps over different fields. The point is that
    $f^\dagger$ preserves ranks and it is exact. Thus if 
\(\varphi: E \to F\) is a morphism of convergent isocrystals on \(X\) such
that \(f^\dagger\varphi = 0\) in \(\operatorname{Isoc}^\dagger(Y/K)\), then we have
\[
f^\dagger\mathcal{I} = \operatorname{Im}(f^\dagger\varphi) = 0.
\]
This implies that $\sI$ has rank 0 as well, i.e. $\sI=0$ and therefore
$\varphi=0$.
\end{proof}

\begin{rmk}
The connectedness of \(X\) is essential. If \(X\) is disconnected, say \(X = X_1 \sqcup X_2\) and \(f(Y)\) lies entirely in \(X_1\), then a non-zero morphism supported on \(X_2\) will pull back to zero and faithfulness fails.
\end{rmk}

\section{Extension of subisocrystals in the overconvergent case}

\begin{prop}
\label{prop:extension-subisocrystal}
Let $U \hookrightarrow X$ be an open immersion of geometrically unibranch
connected separated $k$-schemes of finite type. Suppose that $U$ is dense in
$X$.  Let $E \in$ $(F)$-$\Isoc^{\dagger\dagger}(X/K)$ and let $H$ be a
subisocrystal of $E|_{U}$ in $(F)$-$\Isoc^{\dagger\dagger}(U/K)$. Then $H$ is the restriction to $U$ of a unique subisocrystal of $E$ in $(F)$-$\Isoc^{\dagger\dagger}(X/K)$.
\end{prop}

\begin{proof}
We present the proof for overconvergent isocrystals, as the case of overconvergent $F$-isocrystals is completely analogous. Let
$f\colon X' \to X$ be de Jong's alteration (cf.~\cite[Thm. 4.1]{dJ96}), i.e., $X'$ is smooth connected and $f$ is
proper surjective and generically \'etale. Thus $U'\coloneqq f^{-1}(U)$ is dense in
$X'$.
 Let $E',H'$ denote
$f^*E$ and $f^*H$, respectively. Thanks to \cite[Thm.~5.2.1, Prop.~5.3.1]{Ked07}, the
subisocrystal $H'\hookrightarrow E'|_{U'}$ extends to a unique subisocrystal
$G'\hookrightarrow E'$ on $X'$.

We are going to apply proper descent of overconvergent isocrystals
(cf.~\cite[Thm. 5.1]{Laz22}) to the covering $f$ and the extension $G'\subset
E'$ of $H'\subset E'|_{U'}$, to show that
$G'$ descends to an overconvergent subisocrystal of $E$.

We claim that the map $\pi_{0,f}\colon\pi_0(U'\times_U U')\to
\pi_0(X'\times_X X')$ is surjective. To see that we consider the Stein
factorization $X'\to Y\to X$ of $f$, where the first map has geometrically
connected fibers and the second map is finite surjective. Let $U_Y$ be the
preimage of $U$ in $Y$. Consider the following commutative diagram
\[        
            \begin{tikzpicture}[xscale=2.0,yscale=1.2,bmr/.pic={\draw
                (0,0)--++(-90:2mm)--++(180:2mm);},baseline={([yshift=-.5ex]current bounding box.center)}]
                 \path
                        (0,0)     node (F) {$U_Y\times_UU'$}
                     +(0:1.5)  node (star) {$U_Y\times_UU_Y$}
                     ++(-90:1.5) node (X) {$Y\times_XX'$}
                     +(0:1.5)  node (Y) {$Y\times_XY$}
                     ++(-180:1.5) node (Z) {$X'\times_XX'$}
                     +(90:1.5) node (T) {$U'\times_UU'$};
                 \draw[->] (F)--(star);
                 \draw[->] (F)--(X);
                 \draw[->] (X)--(Y) node[midway,above,scale=.6]{};
                 \draw[->] (star)--(Y);
                 \draw[->] (T)--(Z)
                 node[midway,above,scale=.6]{};
                 \draw[->] (Z) to 
                 node[midway,above,scale=.6]{} (X);
                 \draw[->] (T) to (F);
            \end{tikzpicture}
        \]
The horizontal maps become isomorphisms after applying $\pi_0$. Thus we may assume that $f$ is finite surjective. Now applying \cite[Lemma \href{https://stacks.math.columbia.edu/tag/0F32}{0F32}]{stacks-project}, we find that $U'\times_U U'$ is dense in $X'\times_X X'$, so in particular, $\pi_{0,f}$ is surjective.

Set $R=X'\times_X X'$, $R_U=U'\times_U U'$, and let
$\alpha\colon \mathrm{pr}_1^*E'\xrightarrow{\sim}\mathrm{pr}_2^*E'$
be the canonical descent isomorphism. Denote by
$i_j\colon \mathrm{pr}_j^*G'\hookrightarrow \mathrm{pr}_j^*E'$
the inclusions, and let
$q\colon \mathrm{pr}_2^*E'\to
\mathrm{pr}_2^*E'/\mathrm{pr}_2^*G'$
be the quotient map. Consider the composition
\[
\mathrm{pr}_1^*G'\xrightarrow{i_1}\mathrm{pr}_1^*E'
\xrightarrow{\alpha}\mathrm{pr}_2^*E'
\xrightarrow{q}
\mathrm{pr}_2^*E'/\mathrm{pr}_2^*G'.
\]
Since $G'|_{U'}=H'$, the restriction of $\alpha$ to $R_U$ maps
$\mathrm{pr}_1^*H'$ into $\mathrm{pr}_2^*H'$, and hence the above composite
restricts to zero on $R_U$. As shown above, $R_U$  meets every connected component of $R$. Thus restriction to
$R_U$ is faithful componentwise by Lemma~\ref{lem:pullback-faithful}, and
therefore the composite is zero on $R$. Consequently, $\alpha\circ i_1$
factors through the kernel of $q$, namely through
$i_2\colon \mathrm{pr}_2^*G'\hookrightarrow \mathrm{pr}_2^*E'$. Hence we
obtain a unique morphism
\[
\lambda\colon \mathrm{pr}_1^*G'\longrightarrow \mathrm{pr}_2^*G'
\]
such that $i_2\circ\lambda=\alpha\circ i_1$:
\[
\begin{tikzpicture}[xscale=2.4]
  \node (G1) at (0,0) {$\mathrm{pr}_1^*G'$};
  \node (G2) at (1.2,0) {$\mathrm{pr}_2^*G'$};
  \node (E1) at (0,1) {$\mathrm{pr}_1^*E'$};
  \node (E2) at (1.2,1) {$\mathrm{pr}_2^*E'$};
  \draw[->,dashed] (G1) -- node[anchor=south]{$\lambda$} (G2);
  \draw[->] (E1) -- node[anchor=south]{$\alpha$} (E2);
  \draw[->] (G1) -- node[left]{$\subset$} (E1);
  \draw[->] (G2) -- node[left]{$\subset$} (E2);
\end{tikzpicture}
\]
Applying the same argument to $\alpha^{-1}$ yields a morphism
\[
\mu\colon \mathrm{pr}_2^*G'\longrightarrow \mathrm{pr}_1^*G'
\]
such that $i_1\circ\mu=\alpha^{-1}\circ i_2$. By the uniqueness of such
factorizations, we have $\mu\circ\lambda=\mathrm{id}$ and
$\lambda\circ\mu=\mathrm{id}$; hence $\lambda$ is an isomorphism.

It remains to check the cocycle condition. Let
$T=X'\times_X X'\times_X X'$, with projections
$\mathrm{pr}_i\colon T\to X'$ and
$\mathrm{pr}_{ij}\colon T\to R$. The cocycle condition for $\lambda$ is the
equality
\[
\mathrm{pr}_{13}^*\lambda =
\mathrm{pr}_{23}^*\lambda\circ \mathrm{pr}_{12}^*\lambda
\]
as morphisms $\mathrm{pr}_1^*G'\to \mathrm{pr}_3^*G'$. After composing with
the monomorphism $\mathrm{pr}_3^*G'\hookrightarrow \mathrm{pr}_3^*E'$, this
equality follows from the cocycle condition for $\alpha$. Therefore the
cocycle condition for $\lambda$ holds.

Thus $\lambda$ defines descent data on $G'$ relative to the proper covering
$f$, compatible with the descent data on $E'$. By proper descent
(cf.~\cite[Thm. 5.1]{Laz22}), $G'$ descends to a subisocrystal
$G\subset E$ on $X$. Its restriction to $U$ is $H$ by construction.

For uniqueness, suppose $G_1,G_2\subset E$ both restrict to $H$ on $U$. Then
$f^*G_1$ and $f^*G_2$ both restrict to $H'$ on $U'$; by the uniqueness of the
extension on $X'$, we have $f^*G_1=f^*G_2\subset E'$. By proper descent, this
implies $G_1=G_2$.
\end{proof}

\

\begin{thm}\label{relative gerbe} Let $U \hookrightarrow X$ be an open immersion of geometrically
    unibranch connected separated $k$-schemes of finite type. Suppose that $U$
    is dense in $X$. Then $\Pi^{\Isoc^{\dagger\dagger}}_U\to
    \Pi^{\Isoc^{\dagger\dagger}}_X$ and $\Pi^{F\text{-}\Isoc^{\dagger\dagger}}_U\to
    \Pi^{F\text{-}\Isoc^{\dagger\dagger}}_X$ are quotient maps.
\end{thm}
\begin{proof}
By \cite[Corollaries 1.2 and
    1.3]{Tsu12}, the restriction functor is
fully faithful; in particular, it identifies the endomorphism fields of
the tensor units. By Proposition \ref{prop:extension-subisocrystal}, every subobject of a restricted
object extends to \(X\). The assertion therefore follows from \cite[Lemma 1.6]{Langer-Zhang2025}.
\end{proof}

\section{A counterexample for convergent isocrystals}
\label{sec:counterexample}

The following example shows that Proposition~\ref{prop:extension-subisocrystal} is
false if one replaces ($F$)-$\Isoc^{\dagger\dagger}(X/K)$ by
($F$)-$\Isoc^{\dagger}(X/K)$ or ($F$)-$\Isoc^{}(X/K)$, even when $X$ is smooth.
The example is well known to experts; it appears implicitly in the work of Crew
\cite{Crew87} on unit‑root \(F\)-isocrystals; more
straightforward arguments appear in \cite{Ked07} and \cite{Kedlaya2022}).

\begin{ex}[Family of elliptic curves minus the supersingular fibres]
\label{ex:legendre}
{Let \(k\) be a perfect field of characteristic \(p>0\), and let $K=W(k)[1/p]$. Consider the open modular curve
$X\coloneqq Y(N)$ of level $N\geq 3$ with $p\nmid N$.} Let $U\subset X$ be the ordinary locus,
and let $f\colon E\to X$ be the universal elliptic curve. 
The first relative rigid
cohomology with trivial coefficients \(\mathcal{E} = R^1 f_{{\sc rig} *}
\mathcal{O}_{E/K}\) is an \emph{overconvergent} \(F\)-isocrystal on \(X\) of rank
\(2\) \cite[Thm.~2.1]{Crew87}. Its Frobenius slopes are \(0\) and \(1\) at
each closed point of \(U\). By the theory of unit‑root \(F\)-isocrystals
 \emph{loc.\ cit.}, there exists a rank‑\(1\) convergent
sub‑\(F\)-isocrystal \(\mathcal{F}\subset \mathcal{E}|_U\) on \(U\) (the unit‑root
part), which is unique up to isomorphism.  However, \(\mathcal{F}\) does
\emph{not} extend to any \emph{sub‑\(F\)-isocrystal} of $\sE$ on the whole \(X\).

Assume the contrary, i.e. that \(\mathcal{F}\) extended to a
sub‑\(F\)-isocrystal $\sG$ of \(\mathcal{E}\). Choose any geometric point
$\bar{x}=\Spec(\bar{k})$ of $X\setminus U$. Then $\sG_{\bar{x}}\subset \sE_{\bar{x}}$ is a
rank 1 subobject. At a supersingular point, both Frobenius slopes of \(\mathcal
E_{\bar x}\) are \(\frac{1}{2}\). By the Dieudonné–Manin classification
(cf.~\cite[Thm.~3.2]{Kedlaya2022}), a simple
object of slope \(\frac{a}{b}\), with \(\gcd(a,b)=1\), has rank \(b\). Hence the
rank-two isoclinic object \(\mathcal E_{\bar x}\) of slope \(\frac{1}{2}\) is
simple. Thus it cannot contain a rank‑\(1\) subobject. This
contradicts the existence of \(\mathcal{G}\subset \sE\).

Moreover, the convergent sub‑\(F\)-isocrystal \(\mathcal{F}\) of the
overconvergent $F$-isocrystal \(\mathcal{E}|_U\), regarded, after forgetting
Frobenius, as a subisocrystal of \(\sE|_U\)
in \(\operatorname{Isoc}(U/K)\), does not lift to a sub-isocrystal
of \(\mathcal{E}\) either. Indeed, once $\sG\subset \sE$ is a sub-isocrystal in
$\Isoc(X/K)$ lifting $\sF\subset\sE|_U$,  the composition of
maps
%
%
    \(
    F^*\sG\subset F^*\sE\xrightarrow{\phi_\sE}\sE\longrightarrow \sE/\sG.
    \)
     vanishes over \(U\), hence globally by Lemma
    \ref{lem:pullback-faithful}. Therefore \(\phi_\sE\)
factors through a map \(\phi_\sG\colon F^*\sG\to \sG\). Since \(\phi_\sG|_U\) is an
isomorphism, exactness and faithfulness of restriction show that
\(\phi_\sG\) is an isomorphism. This implies that $\sG\subset\sE$ is
a sub-\(F\)-isocrystal extending the map $\sF\subset \sE|_U$
in $F$-$\Isoc(U/K)$. This contradicts our previous conclusion.

By choosing a connected component of $X$, one sees that the analogue of
Proposition \ref{prop:extension-subisocrystal} for (convergent)
($F$)-isocrystals fails. By taking $k$ to be algebraically closed, one gets the
counterexample claimed in Theorem \ref{main1} (2).
\end{ex}

\section*{Acknowledgements}

{We} would like to thank Kieran Kedlaya for answering some questions concerning \cite{Kedlaya2022}.
{The first author} was partially supported by Polish National Centre (NCN) contract number 2025/59/B/ST1/02168. The second author was supported
by Guangdong Basic and Applied Basic Research Foundation grant 
2025A1515012175.

\printbibliography

@unpublished{Berthelot96,
    author = {Berthelot, Pierre},
    title = {Cohomologie rigide et cohomologie rigide à supports propres. {P}remière partie},
    note = {Prépublication 96-03, Université de Rennes 1, version provisoire 1991,
            available at \url{https://perso.univ-rennes1.fr/pierre.berthelot/}},
    year = {1996},
    month = {January},
}

@misc{BO78,
Author = {Pierre {Berthelot} and Arthur {Ogus}}, 
Title = {{Notes on crystalline cohomology.}}, 
Year = {1978}, 
Language = {English}, 
HowPublished = {{Princeton University Press, Princeton, N.J.}}, 
DOI = {10.1515/9781400867318}, 
MSC2010 = {14F30 14F20 14G10 14-02}, 
Zbl = {0383.14010}
}

@incollection {Crew87,
    AUTHOR = {Crew, Richard},
     TITLE = {{$F$}-isocrystals and {$p$}-adic representations},
 BOOKTITLE = {Algebraic geometry, {B}owdoin, 1985 ({B}runswick, {M}aine,
              1985)},
    SERIES = {Proc. Sympos. Pure Math.},
    VOLUME = {46, Part 2},
     PAGES = {111--138},
 PUBLISHER = {Amer. Math. Soc., Providence, RI},
      YEAR = {1987},
      ISBN = {0-8218-1480-X},
   MRCLASS = {14F30},
  MRNUMBER = {927977},
MRREVIEWER = {Jan\ Stienstra},
       DOI = {10.1090/pspum/046.2/927977},
       URL = {https://doi.org/10.1090/pspum/046.2/927977},
}

@article {dJ96,
    AUTHOR = {de Jong, A. J.},
     TITLE = {Smoothness, semi-stability and alterations},
   JOURNAL = {Inst. Hautes \'Etudes Sci. Publ. Math.},
  FJOURNAL = {Institut des Hautes \'Etudes Scientifiques. Publications
              Math\'ematiques},
    NUMBER = {83},
      YEAR = {1996},
     PAGES = {51--93},
      ISSN = {0073-8301,1618-1913},
   MRCLASS = {14E15 (14B05 14H10)},
  MRNUMBER = {1423020},
MRREVIEWER = {Marko\ Roczen},
       URL = {http://www.numdam.org/item?id=PMIHES_1996__83__51_0},
}

@incollection {drinfeld24,
    AUTHOR = {Drinfeld, Vladimir},
     TITLE = {A stacky approach to crystals},
 BOOKTITLE = {Dialogues between physics and mathematics---{C}. {N}. {Y}ang
              at 100},
     PAGES = {19--47},
 PUBLISHER = {Springer, Cham},
      YEAR = {[2022] \copyright 2022},
      ISBN = {978-3-031-17522-0; 978-3-031-17523-7},
   MRCLASS = {14F30},
  MRNUMBER = {4696502},
MRREVIEWER = {Ariyan\ Javanpeykar},
       DOI = {10.1007/978-3-031-17523-7\_2},
       URL = {https://doi.org/10.1007/978-3-031-17523-7_2},
}

@article{Ked07, 
    title={Semistable reduction for overconvergent $F$-isocrystals I: Unipotence and logarithmic extensions}, 
    volume={143}, 
    DOI={10.1112/S0010437X07002886}, 
    number={5}, 
    journal={Compositio Mathematica}, 
    author={Kedlaya, Kiran S.}, 
    year={2007}, 
    pages={1164–1212}
}

@article {Laz22,
    AUTHOR = {Lazda, Christopher},
     TITLE = {A note on effective descent for overconvergent isocrystals},
   JOURNAL = {J. Number Theory},
  FJOURNAL = {Journal of Number Theory},
    VOLUME = {237},
      YEAR = {2022},
     PAGES = {395--410},
      ISSN = {0022-314X,1096-1658},
   MRCLASS = {14F30 (14G17 14G22)},
  MRNUMBER = {4410031},
MRREVIEWER = {Adolfo\ Quir\'os},
       DOI = {10.1016/j.jnt.2019.09.014},
       URL = {https://doi.org/10.1016/j.jnt.2019.09.014},
}

@ARTICLE{Ogus84,     
Author = {Arthur {Ogus}}, 
Title = {{F-isocrystals and de Rham cohomology. II: Convergent isocrystals.}}, 
FJournal = {{Duke Mathematical Journal}}, 
Journal = {{Duke Math. J.}}, 
ISSN = {0012-7094; 1547-7398/e}, 
Volume = {51}, 
Pages = {765--850}, 
Year = {1984}, 
Publisher = {Duke University Press, Durham, NC; University of North Carolina, Chapel Hill, NC}, 
Language = {English}, 
DOI = {10.1215/S0012-7094-84-05136-6}, 
MSC2010 = {14F30 14F40 14G20 14D99}, 
Zbl = {0584.14008}
 }

@MISC{stacks-project,
    AUTHOR = "Authors, The Stacks Project",
    TITLE = "Stacks Project",
    URL = "https://stacks.math.columbia.edu/"
}

@article {Tsu12,
    AUTHOR = {Tsuzuki, Nobuo},
     TITLE = {A note on the first rigid cohomology group for geometrically
              unibranch varieties},
   JOURNAL = {Rend. Semin. Mat. Univ. Padova},
  FJOURNAL = {Rendiconti del Seminario Matematico della Universit\`a{} di
              Padova. Mathematical Journal of the University of Padua},
    VOLUME = {128},
      YEAR = {2012},
     PAGES = {17--53},
      ISSN = {0041-8994,2240-2926},
   MRCLASS = {14F43 (14G20)},
  MRNUMBER = {3076830},
MRREVIEWER = {Alan\ G. B. Lauder},
       DOI = {10.4171/RSMUP/128-3},
       URL = {https://doi.org/10.4171/RSMUP/128-3},
}

@article {Kedlaya2022,
    AUTHOR = {Kedlaya, Kiran S.},
     TITLE = {Notes on isocrystals},
   JOURNAL = {J. Number Theory},
  FJOURNAL = {Journal of Number Theory},
    VOLUME = {237},
      YEAR = {2022},
     PAGES = {353--394},
      ISSN = {0022-314X,1096-1658},
   MRCLASS = {14F30 (11S37 14F20 14G15)},
  MRNUMBER = {4410030},
MRREVIEWER = {Michel\ Gros},
       DOI = {10.1016/j.jnt.2021.12.004},
       URL = {https://doi.org/10.1016/j.jnt.2021.12.004},
}

@incollection {Fulton-Lazarsfeld1981,
    AUTHOR = {Fulton, William and Lazarsfeld, Robert},
     TITLE = {Connectivity and its applications in algebraic geometry},
 BOOKTITLE = {Algebraic geometry ({C}hicago, {I}ll., 1980)},
    SERIES = {Lecture Notes in Math.},
    VOLUME = {862},
     PAGES = {26--92},
 PUBLISHER = {Springer, Berlin},
      YEAR = {1981},
   MRCLASS = {14-02 (14C99 14E25)},
  MRNUMBER = {644817},
MRREVIEWER = {Knud L\o nsted},
}

@misc{Langer-Zhang2025,
AUTHOR = {Langer, Adrian and Zhang, Lei},
TITLE = {${F}$-divided bundles on normal ${F}$-finite schemes},
YEAR = {2025},
NOTE = {arXiv:2510.10582v1},
}

\end{document}